\documentclass{amsart}
\usepackage{amssymb,amsmath, amsthm,latexsym}
\usepackage{graphics}
\usepackage{amscd}
\usepackage{graphicx}
\newcommand{\cal}[1]{\mathcal{#1}}
\theoremstyle{plain}
\newtheorem*{theo}{Theorem}

\newtheorem{lemma}{Lemma}[section]
\newtheorem{theorem}[lemma]{Theorem}  
\newtheorem{proposition}[lemma]{Proposition}
\newtheorem{corollary}[lemma]{Corollary}
\theoremstyle{definition}

\newtheorem{definition}[lemma]{Definition}
\let\egthree=\phi
\let\phi=\varphi
\let\varphi=\egthree

\begin{document}
\title{Hyperbolicity of the graph of non-separating multicurves}
\author{Ursula Hamenst\"adt}
\thanks{Partially supported by the Hausdorff Center Bonn
and ERC Grant Nb. 10160104}
\date{September 29, 2013}


\begin{abstract} A non-separating multicurve
on a surface $S$ of genus $g\geq 2$ with
$m\geq 0$ punctures is
a multicurve $c$ so that $S-c$ is connected.
For $k\geq 1$ 
define the graph ${\cal N\cal C}(S,k)$
of non-separating $k$-multicurves to be
the graph whose vertices are non-separating multicurves
with $k$ components 
and where two such multicurves are connected
by an edge of length one if they can be realized
disjointly and differ by a single component. 
We show that if $k<g/2+1$ 
then ${\cal N\cal C}(S,k)$
is hyperbolic.
\end{abstract}

\maketitle


\section{Introduction}

The \emph{curve graph} ${\cal C\cal G}$
of an oriented surface $S$ of genus $g\geq 0$
with $m\geq 0$ punctures and
$3g-3+m\geq 2$ is the graph whose vertices 
are isotopy classes of 
essential (i.e. non-contractible and not homotopic 
into a puncture) simple
closed curves on $S$. Two such 
curves are connected by an edge of length one if and only
if they can be realized disjointly. 
The curve graph 
is a locally infinite $\delta$-hyperbolic geodesic metric space of 
infinite diameter \cite{MM99}
for a number $\delta >0$ not depending on the surface
\cite{A12,B12,CRS13,HPW13}.

The \emph{mapping class
group} ${\rm Mod}(S)$ of all isotopy classes
of orientation preserving homeomorphisms of $S$ 
acts on ${\cal C\cal G}$ as
a group of simplicial isometries. 
This action is \emph{coarsely transitive}, i.e. the quotient
of ${\cal C\cal G}$ under this action is a finite graph.
Curve graphs and their geometric properties 
turned out to be an important tool
for the investigation of the geometry of ${\rm Mod}(S)$
\cite{MM00}. 

If the genus $g$ of $S$ is positive then 
for each $k\leq g$ we can 
define another ${\rm Mod}(S)$-graph ${\cal N\cal C}(S,k)$
as follows. Vertices of ${\cal N\cal C}(S,k)$ are
\emph{non-separating $k$-multicurves}, i.e. 
multicurves $\nu$ consisting of $k$ components such that
$S-\nu$ is connected. 
Two such multicurves are connected by an edge of length one
if they can be realized disjointly and differ by a single 
component. The mapping class group of $S$ acts
coarsely transitively as a group of simplicial isometries
on the graph of non-separating $k$-multicurves. 
In fact, the action is transitive on 
vertices. Note that ${\cal N\cal C}(S,1)$ is just the complete
subgraph of ${\cal C\cal G}$ whose vertex set consists
of all non-separating simple closed curves in $S$.

The goal of this note is to show

\begin{theo}\label{cut}
For $g\geq 2$ and 
$k<g/2+1$ the graph ${\cal N\cal C}(S,k)$ of non-separating
$k$-multicurves is hyperbolic.
\end{theo}

For the proof of the theorem, we adopt a strategy
from \cite{H13}. Namely, we begin with showing 
that for $g\geq 2$ the graph ${\cal N\cal C}(S,1)$ 
is hyperbolic. This is easy if 
$S$ has at most one puncture, in fact 
in this 
case the inclusion map ${\cal N\cal C}(S,1)\to {\cal C\cal G}$ 
is a quasi-isometry (see Section 3). 
If $S$ has at least two punctures
then this inclusion is not a quasi-isometry any more. 
In this case we apply a 
tool from \cite{H13}. This tool 
is also used in Section 4 to successively add components to 
the multicurve until the number $k<g/2+1$ of components
is reached.

We summarize the results 
from \cite{H13} which we need in Section 2.
At the end of this note we indicate an example
indicated to us by
Tarik Aougab and Saul Schleimer which shows 
that the strict
bound $g/2+1$ for the number of components of the
multicurve in Theorem \ref{cut} is sharp.

\bigskip

{\bf Acknowledgement:} This work was carried out
while the author visited the Institute for Pure
and Applied Mathematics in Los Angeles.
In a first version of this paper, I erraneously misstated
the range of the number of components of 
multicurves for which Theorem \ref{cut} is valid.  
I am grateful to Tarik Aougab and 
Saul Schleimer for pointing out this error to me.

\section{Hyperbolic extensions of hyperbolic graphs}

In this section we consider 
any (not necessarily locally finite) metric graph $({\cal G},d)$
(i.e. edges have length one). Let ${\cal C}$ be any 
finite, countable or empty index set. 
For a given family
${\cal H}=\{H_c\mid c\in {\cal C}\}$ of complete connected
subgraphs of ${\cal G}$ define the \emph{${\cal H}$-electrification}
of ${\cal G}$ to be the 
metric graph $({\cal E\cal G},d_{\cal E})$ which is obtained
from ${\cal G}$ by adding vertices and edges as follows.

For each $c\in {\cal C}$ there is a unique vertex 
$v_c\in {\cal E\cal G}-{\cal G}$. This vertex
is connected with each of the vertices of $H_c$ by a single
edge of length one, and it is not connected with any other vertex.

\begin{definition}\label{rconvex} For a number $r>0$ 
the family ${\cal H}$ is called \emph{$r$-bounded} if 
for $c\not= d\in {\cal C}$ the intersection
$H_c\cap H_d$ has diameter at most $r$ where 
the diameter is taken with respect to the intrinsic
path metric on $H_c$ and $H_d$.
\end{definition}

A family which is $r$-bounded for some $r>0$ is simply
called \emph{bounded}.

In the sequel all parametrized 
paths $\gamma$ in ${\cal G}$ or ${\cal E\cal G}$
are supposed to be \emph{simplicial}. This means that
they are defined on a closed connected 
subset of the reals whose finite endpoints (if any) are
integers. We require that
the image of every integer is a vertex, and that the 
restriction to 
an integral interval $[k,k+1]$ either is homeomorphism
onto an edge, or it is constant.
In particular, simplicial paths are continuous.

Call a simplicial path $\gamma$ in ${\cal E\cal G}$ 
\emph{efficient} if for every $c\in {\cal C}$
we have $\gamma(k)=v_c$ for at most one integer $k$.
Note that if $\gamma$ is an efficient 
simplicial path in ${\cal E\cal G}$
which passes through $\gamma(k)=v_c$ for some $c\in {\cal C}$
then $\gamma(k-1)\in H_c,\gamma(k+1)\in H_c$.
This is true because the vertex $v_c\in {\cal E\cal G}$ is only
connected with vertices in $H_c$ by an edge.

For a number $L>1$, an \emph{$L$-quasi-geodesic} in 
${\cal E\cal G}$ is a path 
$\gamma:[a,b]\to {\cal E\cal G}$ such that for all 
$a\leq s <t\leq b$ we have
\[\vert t-s\vert/L-L\leq d(\gamma(s),\gamma(t))\leq 
L\vert t-s\vert +L.\]
In slight deviation from this standard definition, 
throughout we
require in the sequel that all quasi-geodesics are
simplicial, in particular, they are continuous.
We will often but not always state this explicitly.

\begin{definition}\label{convexpenetration}
The family ${\cal H}$ has 
the \emph{bounded penetration property} if it
is $r$-bounded for some $r>0$ and if 
for every $L>0$ there is a number
$p(L)>2r$ with the following property.
Let $\gamma$ be an efficient simplicial 
$L$-quasi-geodesic in 
${\cal E\cal G}$, let $c\in {\cal C}$ and 
let $k\in \mathbb{Z}$ be such that
$\gamma(k)=v_c$. If the distance in $H_c$ between
$\gamma(k-1)$ and $\gamma(k+1)$ is at least $p(L)$
then every efficient simplicial $L$-quasi-geodesic $\gamma^\prime$
in ${\cal E\cal G}$ with the same endpoints as $\gamma$
passes through $v_c$. Moreover, if 
$k^\prime\in \mathbb{Z}$ is 
such that $\gamma^\prime(k^\prime)=v_c$
them the distance in $H_c$ between 
$\gamma(k-1),\gamma^\prime(k^\prime-1)$ 
and between $\gamma(k+1),\gamma^\prime(k^\prime+1)$  
is at most $p(L)$.
\end{definition}

Let ${\cal H}$ be as in Definition \ref{convexpenetration}.
Define an \emph{enlargement} $\hat \gamma$ of an efficient
 simplicial $L$-quasi-geodesic $\gamma:[0,n]\to {\cal E\cal G}$ 
with endpoints $\gamma(0),\gamma(n)\in {\cal G}$ as
follows. Let $0<k_1<\dots <k_s< n$ be those points
such that $\gamma(k_i)=v_{c_i}$ for some $c_i\in {\cal C}$.
Then $\gamma(k_i-1),\gamma(k_i+1)\in H_{c_i}$. 
For each $i\leq s$ replace 
$\gamma[k_i-1,k_i+1]$ by a simplicial geodesic in 
the graph $H_{c_i}$ with
the same endpoints. Note that since we require
that the endpoints of $\gamma$ are vertices in ${\cal G}$,
an enlargement of $\gamma$ is a path with the same endpoints.

For a number $k>0$ define a subset $Z$ of the 
metric graph ${\cal G}$ to be 
\emph{$k$-quasi-convex}
if any geodesic with both endpoints in $Z$ is contained in the
$k$-neighborhood of $Z$. In particular, up to perhaps increasing
the number $k$, any two points in $Z$ can be connected
in $Z$ by a (not necessarily continuous) path which is 
a $k$-quasi-geodesic in ${\cal G}$.

In Section 5 of \cite{H13} the following is shown.

\begin{theorem}\label{hypextension}
Let ${\cal G}$ be a metric graph and let 
${\cal H}=\{H_c\mid c\in {\cal C}\}$ 
be a bounded family of 
complete connected subgraphs of ${\cal G}$. 
Assume that the following conditions are satisfied.
\begin{enumerate}
\item There is a number $\delta >0$ such that 
each of the graphs $H_c$ is $\delta$-hyperbolic.
\item The ${\cal H}$-electrification ${\cal E\cal G}$
of ${\cal G}$ is hyperbolic.
\item ${\cal H}$ has the bounded  
penetration property.
\end{enumerate}
Then ${\cal G}$ is hyperbolic. There is a number $L>1$ such that
enlargements of 
geodesics in ${\cal E\cal G}$ are $L$-quasi-geodesics in ${\cal G}$.
The subgraphs $H_c$ are uniformly quasi-convex.
\end{theorem}

In fact, although this was not stated explicitly, 
one obtains that the graph ${\cal G}$ is 
$\delta^\prime$-hyperbolic for a number $\delta^\prime>0$ only
depending on the hyperbolicity constant for ${\cal E\cal G}$,
the common hyperbolicity constant $\delta$ for the subgraphs $H_c$ and
the constants which enter in the bounded penetration
property.

\section{Hyperbolicity of the graph of non-separating curves}

In this section we consider an arbitrary surface $S$ of 
genus $g\geq 2$ with $m\geq 0$ punctures. 
Let ${\cal C\cal G}$ be the curve graph of 
$S$ and let ${\cal N\cal C}(S,1)$ be the complete
subgraph of ${\cal C\cal G}$ 
whose vertex set consists of non-separating curves.
The goal of this section is to show

\begin{proposition}\label{curves}
The graph ${\cal N \cal C}(S,1)$ is hyperbolic.
\end{proposition}

\noindent
{\bf Example:} If $S$ is a surface of genus $g=1$ 
then any two disjoint non-separating simple closed
curves in $S$ are homotopic after closing the punctures 
and the graph ${\cal N\cal C}(S,1)$  
is not connected.

\bigskip

Define a properly embedded connected
incompressible subsurface $X$ of 
$S$ to be \emph{thick} if the genus of $X$ equals $g$. This
is equivalent to stating that 
each of the boundary circles of $X$ 
is separating
in $S$ and that moreover there is no non-separating simple
closed curve in $S$ which is contained in $S-X$. 
Observe that the only thick subsurface of a surface $S$ with
at most one puncture is $S$ itself.

If $X\subset S$ is thick then each component of $S-X$
is a bordered punctured sphere with connected boundary.
If we collapse each boundary circle of $X$ to
a puncture then we can view $X$ as a surface of finite type whose 
genus equals the genus of $S$. 
In particular, we can look at thick
subsurfaces of $X$. However, thick subsurfaces of $X$
are precisely the
thick subsurfaces of $S$ which are contained in $X$.

For a thick subsurface $X$ of $S$ and for 
$p\geq 1$ define a graph ${\cal A}(X,p)$ as follows.
Vertices of ${\cal A}(X,p)$ are non-separating simple closed curves
in $X$. Two such vertices $c,d$ are connected by an edge
of length one if either they are disjoint or if they
are both contained in a proper thick subsurface $Y$ of $X$ of 
Euler characteristic $\chi(X)+p$. Note that if 
$p\geq  -\chi(X)-2g+2$ then ${\cal A}(X,p)={\cal N\cal C}(X,1)$.

Recall that for a number $L\geq1$ 
two geodesic metric spaces $Y,Z$ are
\emph{$L$-quasi-isometric} if 
there is a map $F:Y\to Z$ so that
\[d(x,y)/L-L\leq d(Fx,Fy)\leq Ld(x,y)+L\,\forall x,y\in Y\]
and that for all $z\in Z$ there is some 
$y\in Y$ with $d(Fy,z)\leq L$.
In general, quasi-isometries are not continuous. 
A map $F:Y\to Z$ is called \emph{coarsely
$L$-Lipschitz} if $d(Fx,Fy)\leq Ld(x,y)+L$ for all 
$x,y\in Y$.

Let ${\cal C\cal G}(X)$ be the curve graph of $X$.

\begin{lemma}\label{inductbegin}
For every thick subsurface $X$ of $S$ 
the vertex inclusion extends to a $2$-quasi-isometry
${\cal A}(X,1)\to {\cal C\cal G}(X)$.
\end{lemma}
\begin{proof} Since two simple closed curves which 
are contained in a proper thick subsurface $Y$ of $X$
are disjoint from a boundary circle of $Y$ which is essential in $X$,
the vertex inclusion extends to a coarsely $2$-Lipschitz map 
${\cal A}(X,1)\to {\cal C\cal G}(X)$.
Thus it suffices to show that the distance
in ${\cal A}(X,1)$ between any two non-separating 
simple closed curves
does not exceed twice their distance in ${\cal C\cal G}(X)$.

To this end let $\gamma:[0,n]\to {\cal C\cal G}(X)$ be a simplicial geodesic
connecting two non-separating simple closed 
curves $\gamma(0),\gamma(n)$.  
We construct first a simplicial 
geodesic $\tilde \gamma$ in ${\cal C\cal G}(X)$ with 
the same endpoints such that for each $i$, the curve
$\tilde \gamma(i)$ either is non-separating or 
it decomposes $X$ into a thick subsurface 
of Euler characteristic $\chi(X)+1$ and a 
three-holed sphere. Call such a simple closed curve
(with either
of these two properties) \emph{admissible}
in the sequel.

For the construction of $\tilde \gamma$ replace 
first each of the vertices $\gamma(2i)$ with even 
parameter $0<2i<n$ by an
admissible curve.  
Namely, if $\gamma(2i)$ is not admissible
then $\gamma(2i)$ decomposes 
$X$ into two surfaces $X_1,X_2$
which are different from three holed spheres.

If $\gamma(2i-1),\gamma(2i+1)$ are contained in distinct
components of $S-\gamma(2i)$ then they are disjoint and hence
they are connected in ${\cal C\cal G}(X)$ by an edge.
This implies that we can shorten $\gamma$ with fixed endpoints.
Since $\gamma$ is length minimizing this is impossible.

Thus  
$\gamma(2i-1),\gamma(2i+1)$ are contained in the same component of 
$X-\gamma(2i)$, say in $X_1$. Then
$X_2=X-X_1$ either has positive genus and hence contains a non-separating
curve, or it is a sphere with at least four
holes and contains an admissible separating curve. Thus there is an admissible 
curve $\tilde \gamma(2i)\subset X_2$, and this curve is disjoint from
$\gamma(2i-1)\cup \gamma(2i+1)$. Replace $\gamma(2i)$ by
$\tilde \gamma(2i)$.
This process leaves the points $\gamma(2i+1)$ 
with odd parameter unchanged. 

In a second step, replace 
with the same construction each of 
the points $\gamma(2i+1)$ with odd parameter by an admissible curve. 
Let $\tilde \gamma:[0,n]\to {\cal C\cal G}(X)$ be the resulting
simplicial geodesic. The image of every vertex is admissible. 

The geodesic $\tilde \gamma$ is now modified as follows.
Replace each edge 
$\tilde \gamma[i,i+1]$ connecting two separating admissible simple closed
curves $\tilde \gamma(i),\tilde \gamma(i+1)$ 
by an edge path in ${\cal C\cal G}(X)$ 
of length $2$ with the same endpoints so that
the middle vertex is a non-separating simple closed curve.
This is possible because if 
$c_1,c_2$ are two disjoint separating 
admissible curves then $c_1\cup c_2$ is disjoint
from some non-separating simple closed curve in $X$. 
The length of the resulting path $\hat \gamma$ is at most
twice the length of $\gamma$. 

The path $\hat\gamma$ can be viewed as a path in ${\cal A}(X,1)$ by
simply erasing all vertices which are 
separating admissible simple closed curves.
Namely, each such vertex $v$ is the boundary circle of a thick
subsurface $Y$ of $X$ of Euler characteristic $\chi(X)+1$.
The two adjacent vertices are  
non-separating simple
closed curves contained in $Y$.
Thus by the definition of ${\cal A}(X,1)$, 
these curves are connected in ${\cal A}(X,1)$ by an edge.
This shows that the endpoints of $\gamma$ are connected in 
${\cal A}(X,1)$ by a path whose length does not exceed twice the distance 
in ${\cal C\cal G}(X)$ between the endpoints.
\end{proof}

Since a surface with at most one puncture does not
admit any proper thick subsurface we obtain as an 
immediate corollary

\begin{corollary}\label{onepuncture}
If $S$ has at most one puncture then
the inclusion ${\cal N\cal C}(S,1)\to 
{\cal C\cal G}$ is a $1$-quasi-isometry.
\end{corollary}

Write ${\cal A}(p)={\cal A}(S,p)$.
Our goal is to use Lemma \ref{inductbegin} and induction on $p$
to show that ${\cal A}(p)$ is hyperbolic for all $p$.
Since ${\cal A}(p)={\cal N\cal C}(S,1)$ for $p\geq-\chi(S)-2g+2$,
this then shows Proposition \ref{curves}.

Let now $p-1\geq 1$ and let $X$ be a thick subsurface
of $S$ such that $\chi(X)=\chi(S)+p-1$.
Let $H_{X}$ be the complete subgraph of 
${\cal A}(p)$ whose vertex set consists
of all non-separating simple closed curves
which are contained in $X$.
Let ${\cal H}=\{H_X\mid X\}$. 
Our goal is to apply 
Theorem \ref{hypextension}.
to the graph ${\cal A}(p)$ and its ${\cal H}$-electrification.
The next easy observation
is the basic setup for the induction step.

\begin{lemma}\label{helec}
${\cal A}(p-1)$ is $2$-quasi-isometric to the
${\cal H}$-electrification of ${\cal A}(p)$. 
\end{lemma}
\begin{proof} Let ${\cal E}$ be the
${\cal H}$-electrification of ${\cal A}(p)$. 
Let $c,d$ be any two simple closed curves which are connected
in ${\cal A}(p-1)$ by an edge. Then either $c,d$ are disjoint
and hence connected in ${\cal A}(p)$ by an edge, or 
$c,d$ are contained in a thick subsurface $X$ of $S$ of 
Euler characteristic $\chi(X)=\chi(S)+p-1$. Thus
$c,d$ are vertices in $H_X$ and hence  
the distance between $c,d$ in 
${\cal E}$ is at most two. This shows that
the vertex inclusion ${\cal A}(p-1)\to {\cal E}$ is 
two-Lipschitz. 

That this is in fact a 2-quasi-isometry follows 
from the observation that ${\cal A}(p)$ is obtained
from ${\cal A}(p-1)$ by deleting some edges. Moreover, 
the endpoints of an embedded
simplicial path in ${\cal E}$ of length 2 whose 
midpoint is a special vertex not contained in ${\cal A}(p)$
are non-separating simple closed curves 
which are contained in 
a thick subsurface $X$ of $S$ of Euler characteristic 
$\chi(S)+p-1$ and henc they 
are connected by an edge in ${\cal A}(p-1)$.
\end{proof}

Our goal is now to check that the family 
${\cal H}=\{H_X\mid X\}$ has the
properties stated in Theorem \ref{hypextension}.
The following lemma together with Lemma \ref{inductbegin} 
implies that the graphs $H_X$ are
$\delta$-hyperbolic for a universal constant $\delta >0$.

\begin{lemma}\label{identify} 
$H_{X}$ is isometric to ${\cal A}(X,1)$.
\end{lemma}
\begin{proof} Let $X$ be a thick subsurface of $S$ 
of Euler characteristic
$\chi(S)+p-1$.  
If $c_1,c_2$ are two non-separating simple closed curves 
contained in $X$ then $c_1,c_2$ 
are connected in ${\cal A}(p)$ by an edge if
either $c_1,c_2$ are disjoint or if $c_1,c_2$ are contained in 
a thick subsurface $X_0$ of $S$ of Euler characteristic
$\chi(X_0)=\chi(S)+p=\chi(X)+1$. 

Now the thick subsurface $X_0$ can be chosen to be
contained in $X$. Namely, any thick subsurface of $S$ 
of Euler characteristic $\chi(S)+\ell$ $(\ell\geq 0)$ can
be described as the complement in $S$ of 
a small neighborhood of 
an embedded forest in 
$S$ (i.e. an embedded possibly
disconnected graph with no cycles) with $\ell$ edges 
whose vertices are the punctures of $S$. 

Assume for the moment 
that there is a forest $G$ defining $X$ which is the union of 
a connected component $\hat G$ and isolated points. 
Then $\hat G$ has precisely $p$ vertices and $p-1$ edges
where $p=\chi(X)-\chi(S)$.
Since the graph $G_0$ defining $X_0$ has $p$ edges, 
there is at least one puncture $x$
of $S$ which is not contained in $\hat G$ and which 
is the endpoint of an edge $e$ of  $G_0$.
If up to homotopy with fixed endpoints,
$e$ intersects $\hat G$ at most at 
the second endpoint
then the complement of a small neighborhood of 
$\hat G\cup e$ is a thick subsurface $Y$ of $X$
of Euler characteristic $\chi(S)+p$ 
which contains $c_1\cup c_2$. This is  
what we wanted to show.

If $e$ intersects $\hat G$ in an interior point which can
not be removed with a homotopy of $e$ with fixed endpoitns, 
let $e_0$ be the subarc 
of $e$ with endpoints $x$ and the first intersection point with 
$\hat G$. Concatenation of $e_0$ with 
a subarc of an edge of 
$\hat G$ and modification of the resulting arc with a 
small homotopy with fixed endpoints 
yields an embedded are $\hat e$ in $S$ whose interior is 
disjoint from the interior of 
$\hat G$ so that $\hat G\cup \hat e$
defines a thick subsurface $Y$ of $X$ of Euler characteristic
$\chi(S)+p$ disjoint from $c_1\cup c_2$.

The general case is treated in the same way. Namely, 
there is at least one edge $e$ of $G_0$ which
connects two distinct connected components of $G$, and the
argument above can be applied to the edge $e$.

As a consequence, 
$c_1,c_2$ are connected
by an edge in ${\cal A}(X,1)$ which is what we wanted to show.
\end{proof}

\begin{lemma}\label{rbounded}
The family of subgraphs $H_{X}$ of ${\cal A}(p)$ 
where $X$ runs through the thick subsurfaces of $S$ 
of Euler characteristic $\chi(S)+p-1$ 
is bounded. 
\end{lemma}
\begin{proof} Let $X,Y$ be two thick subsurfaces of $S$ of 
Euler characteristic $\chi(S)+p-1$. 
If $X\not=Y$ then up to homotopy, 
$X\cap Y$ is a (possibly disconnected) subsurface of $X$
whose Euler characteristic is strictly bigger than the
Euler characteristic of $X$. In particular, the diameter
in the curve graph of $X$ of the
set of simple closed curves contained in 
$X\cap Y$ is uniformly bounded.
Thus the lemma follows 
from Lemma \ref{inductbegin} and Lemma \ref{identify}.
\end{proof}

The proof of the bounded penetration property is more involved.
To this end recall from \cite{MM00} 
that for every proper connected subsurface $X$ of $S$
there is a \emph{subsurface projection} $\pi_X$ of 
${\cal C\cal G}$ into the subsets of the arc and curve graph
of $X$. This projection   
associates to a simple closed curve $c$ in $S$ which is not disjoint
from $X$ 
the intersection components $\pi_X(c)$ 
of $c$ with $X$, viewed as 
a subset of the arc and curve graph of $X$. The diameter
of the image is at most one.
If $c$ is disjoint from $X$ then this projection is empty.
The arc and curve graph of $X$ is $2$-quasi-isometric to the curve
graph of $X$ (see \cite{MM00}).

Recall that every vertex of any of the graphs ${\cal A}(p)$ 
$(p\geq 1)$ is a non-separating simple closed curve in $S$.
By definition of a thick subsurface of $S$, 
for any such curve $c$ and every thick subsurface $Z$ of $S$ we have 
$\pi_Z(c)\not=\emptyset$. This fact will be used throughout
the remainder of this section.

We need the following result from \cite{MM00} (in the 
version formulated in Lemma 6.5 of \cite{H13}).

\begin{proposition}\label{projection}
For every number $L>1$ there is a number $\xi(L)>0$ with the
following property. Let $Y$ be a proper connected subsurface of $S$
and let $\gamma$ be a simplicial path
in ${\cal C\cal G}$ which is an $L$-quasi-geodesic.
If $\pi_Y(v)\not=\emptyset$ for every vertex $v$ on $\gamma$ then
\[{\rm diam} \pi_Y(\gamma)<\xi(L).\]
\end{proposition}

If $\gamma:[0,n]\to {\cal A}(S,1)$ is any geodesic then for all $j$,
the curves $\gamma(j)$ and 
$\gamma(j+1)$ either are disjoint and hence connected
in ${\cal C\cal G}$ by an edge, or they are 
contained in a common thick subsurface $Y$ 
of $S$ of Euler characteristic
$\chi(S)+1$. In the second case replace the edge 
$\gamma[j,j+1]$  by an edge path in ${\cal C\cal G}$ of length two
connecting the same endpoints which passes through an 
essential simple closed curve in
the complement of $Y$. We call $\tilde \gamma$ a 
\emph{canonical modification}
of $ \gamma$. By Lemma \ref{inductbegin} and its proof,
$\tilde\gamma$ is a simplicial path
in ${\cal C\cal G}$ which is a $2$-quasi-geodesic.

We now define a family of geodesics in ${\cal A}(S,1)$ which
serve as substitutes for the tight geodesics
as introduced in \cite{MM00}.
Namely, for numbers $\kappa >0,p\geq 1$ define a simplicial 
path $\zeta:[0,k]\to {\cal A}(S,1)$ to be 
\emph{($\kappa$,p)-good}
if the following holds true. Let $X\subset S$ be
any thick subsurface of Euler characteristic
$\chi(X)\geq \chi(S)+p$; then there is a number
$u=u(X)\in [0,k)$ with the following property.
\begin{enumerate}
\item For every $j\leq u$,
${\rm diam}(\pi_X(\zeta(0)\cup\zeta(j)))\leq \kappa.$
\item For every $j>u$,
${\rm diam}(\pi_X(\zeta(j)\cup \zeta(k)))\leq \kappa.$
\end{enumerate}
Thus in a good simplicial path, big subsurface projections
can be explicitly localized.

We use Proposition \ref{projection} to show

\begin{lemma}\label{goodgeo}
There is a number $\kappa_1>0$ such that
any two vertices in ${\cal A}(S,1)$ can be connected by a 
$(\kappa_1,1)$-good geodesic. 
\end{lemma}
\begin{proof}
Let $c_1,c_2$ be non-separating simple closed curves 
and let $\gamma:[0,k]\to {\cal A}(S,1)$ be a simplicial geodesic
connecting $c_1$ to $c_2$, with canonical modification
$\tilde \gamma:[0,\tilde k]\to {\cal C\cal G}$.

Let $\ell_1<\dots<\ell_s$ be such that for each $i$ the curves  
$\tilde\gamma(\ell_i),\tilde\gamma(\ell_i+2)$ are both separating and 
such that the subsurface of $S$ filled by  
$\tilde \gamma(\ell_i)\cup \tilde \gamma(\ell_i+2)$ 
(i.e. the smallest subsurface of $S$ which contains 
$\tilde\gamma(\ell_i)\cup \tilde \gamma(\ell_i+2)$) 
is a holed sphere whose complement $Z$ in $S$ is thick
(we may have $s=0$, i.e. there may not be such a pair of vertices).
Then $\tilde\gamma(\ell+1)$ is a non-separating simple closed
curve contained in $Z$.
Let $\tilde \gamma_1(\ell_i+1)$ be a non-separating simple closed curve
contained in $Z$ which is contained in the one-neighborhood of 
the subsurface projection $\pi_Z(c_2)$ of $c_2$. That the
subsurface projection $\pi_Z(c_2)$ is not empty follows since
$c_2$ is non-separating and hence can not be contained in 
$S-Z$.

Replace $\tilde\gamma(\ell+1)$ by 
$\tilde \gamma_1(\ell+1)$.
The simplicial path $\tilde \gamma_1$ constructed in this way
is a canoncial modification of a geodesic $\gamma_1$
in ${\cal A}(S,1)$ connecting $c_1$ to $c_2$.
We claim that $\gamma_1$ is a $(\xi(2),1)$-good geodesic
in ${\cal A}(S,1)$ where $\xi(2)>0$ is as in 
Proposition \ref{projection}.

Namely, if $Z$ is an arbitrary thick subsurface of $S$ then
since $\gamma_1$ is a geodesic in ${\cal A}(S,1)$.
there are at most two parameters $k,k+\iota$ (here $\iota=0$ or 
$\iota=2$) 
such that $\tilde \gamma_1(k),\tilde \gamma_1(k+\iota)$ is disjoint
from $Z$. Since $\tilde \gamma_1$ is a $2$-quasi-geodesic in 
${\cal C\cal G}$, 
if there is at most one such point 
(which is in particular the case if the Euler
characteristic of $Z$ equals $\chi(S)+1$, i.e. if there is a unique
essential curve disjoint from $Z$) 
then the properties (1),(2) for $Z$ with $\kappa=\xi(2)$ 
are immediate from Lemma \ref{inductbegin}
and Proposition \ref{projection}.
Otherwise the property follows from the construction of $\gamma_1$
and the fact that for subsurfaces $X\subset Y\subset S$
and any simple closed curve $c$ we have
$\pi_X(c)=\pi_X(\pi_Y(c))$ (with a small abuse of notation).
\end{proof}

We use Proposition \ref{projection} and Lemma \ref{goodgeo} to 
define a \emph{level $p$ hierarchy path} in  
${\cal A}(p)$ connecting two non-separating simple closed
curves $c_1,c_2$ as follows. 
The starting point is a $(\kappa_1,1)$-good geodesic 
$\gamma:[0,k]\to {\cal A}(S,1)$.
For any $j$ so that the curves $\gamma(j),\gamma(j+1)$ are 
not disjoint there is a thick subsurface 
$Y_j$ of Euler characteristic $\chi(Y_j)=\chi(S)+1$ 
so that $\gamma(j),\gamma(j+1)\subset Y_j$.
Replace the edge $\gamma[j,j+1]$ by a simplicial $(\kappa_1,1)$-good geodesic
in ${\cal A}(Y_j,1)$ with the same endpoints. 
The resulting path is an edge-path in the
subgraph ${\cal A}(2)$ of ${\cal A}(1)$.
Proceed inductively and construct in $p$ such steps a simplicial path
in ${\cal A}(p)\subset {\cal A}(1)$ connecting $c_1$ to $c_2$ which we call
a \emph{level $p$ hierarchy path}.

\begin{lemma}\label{good2}
For every $p\geq 1$ there is a number $\kappa_p>0$ such that
a level $p$ hierarchy path in ${\cal A}(p)$ is $(\kappa_p,p)$-good.
\end{lemma}
\begin{proof} We proceed by induction on $p$. The case $p=1$ 
follows from the definition of a hierarchy path and
Lemma \ref{goodgeo}. Thus assume that the lemma holds true for
all $p-1\geq 1$. 

Let $\gamma:[0,n]\to {\cal A}(p)$ be a level 
$p$ hierarchy path. The construction of  
$\gamma$ is as follows. There is 
a level $p-1$ hierarchy path 
$\zeta:[0,s]\to {\cal A}(p-1)$, and  there are 
numbers $0\leq \tau_1<\dots <\tau_q<s$ such that
for each $i$, the edge $\zeta[\tau_i,\tau_{i}+1]$ 
connects two non-separating simple closed curves
which are contained in a thick subsurface $Z_i$ of $S$ of 
Euler characteristic $\chi(S)+p-1$. For 
$\ell\not\in\{\tau_1,\dots,\tau_q\}$, the
simple closed curves $\zeta(\ell),\zeta(\ell+1)$
are disjoint. 
The hierarchy path $\gamma$ is obtained 
from $\zeta$ by replacing each of the
edges $\zeta[\tau_i,\tau_i+1]$ by a $(\kappa_1,1)$-good geodesic in 
${\cal A}(Z_i,1)$ with the same endpoints.

By induction hypothesis, $\zeta$ is $(\kappa_{p-1},p-1)$-good for 
a number $\kappa_{p-1}>1$ not depending on $\zeta$.
Thus for any thick subsurface $Z$ of $S$ of Euler characteristic
$\chi(Z)\geq \chi(S)+p$ 
there is a number $u\in [0,s]$ so that
for all $j\leq u$ we have ${\rm diam}(\pi_Z(\zeta(0)\cup\zeta(j)))
\leq \kappa_{p-1}$ 
and similarly for $j\geq u+1$.

Let now $i>0$ be such that $\tau_i<u$.
There is a subarc $\rho$ of $\gamma$ which is a
$(\kappa_1,1)$-good geodesic in ${\cal A}(Z_i,1)$ 
connecting 
$\zeta(\tau_i)$ to $\zeta(\tau_i+1)$. 
By the definition of a $(\kappa_1,1)$-good 
geodesic in ${\cal A}(Z_i,1)$, since 
\begin{align}{\rm diam}(\pi_Z(\zeta(\tau_i)\cup\zeta(\tau_i+1)))& \leq
{\rm diam}(\pi_Z(\zeta(0)\cup \zeta(\tau_i)))+
{\rm diam}(\pi_Z(\zeta(0)\cup \zeta(\tau_i+1)))\notag\\\leq  
2\kappa_{p-1},\notag\end{align} 
for each vertex $\rho(t)$ on the geodesic $\rho$ 
the diameter of the subsurface projection
$\pi_Z(\zeta(\tau_i)\cup \rho(t))$ does not exceed 
$2\kappa_{p-1}+\kappa_1$. Then for each $t$ we have
\[{\rm diam}(\pi_Z(\zeta(0)\cup\rho(t)))\leq 
3\kappa_{p-1}+\kappa_1=\kappa_p.\]
This argument is also valid for $\tau_i>u$.

Finally if $\tau_i=u$ then we can apply the same
reasoning as before to the $\kappa_1$-good geodesic 
$\rho$ and obtain the statement of the lemma.
\end{proof}

\bigskip

\textit{Proof of Proposition \ref{curves}:}
By Lemma \ref{onepuncture}, if $S$ has at most
one puncture then the inclusion 
${\cal N\cal C}(S,1)\to {\cal C\cal G}$ is a quasi-isometry.

If the number of punctures is at least two then 
we show by induction on $p$ 
the following.
\begin{enumerate}
\item[a)] The graph ${\cal A}(p)$ is
hyperbolic.
\item[b)]  Level $p$ hierarchy paths are uniform
quasi-geodesics in ${\cal A}(p)$.
\item[c)] For every $L>1$ there is a number $\xi(L,p)>0$ with the 
following property.  Let
$\gamma$ be a simplicial path in ${\cal A}(p)$ which is
an $L$-quasi-geodesic, and  let $\tilde \gamma$ be the canonical
modification of $\gamma$. If 
 $Y$ is a thick subsurface of $S$
of Euler characteristic $\chi(Y)\geq \chi(S)+p$,
and if $\pi_Y(v)\not=\emptyset$ for every
vertex $v$ on $\tilde\gamma$ then ${\rm diam}\pi_Y(\gamma)<\xi(L,p)$.
\end{enumerate}

The case $p=1$ follows from Lemma \ref{inductbegin},
Proposition \ref{projection} and the definition of 
a canonical modification of a simplicial path in 
${\cal A}(S,1)$..
Assume that the claim holds true for $p-1\geq 1$. 

For a thick subsurface
$X$ of Euler characteristic $\chi(X)=\chi(S)+p-1$ let 
as before $H_{X}$ be the complete subgraph of ${\cal A}(p)$ 
whose vertex set 
consists of all non-separating 
simple closed curves contained in $X$,
and let ${\cal H}=\{H_X\mid X\}$.
By Lemma \ref{helec}, ${\cal A}(p-1)$ is $2$-quasi-isometric
to the ${\cal H}$-electrification of ${\cal A}(p)$.
Moreover by construction, level $p$ hierarchy paths
are enlargements of level $p-1$ hierarchy paths. 
Therefore by the induction hypothesis, to establish
properties a),b) above for $p$ it suffices
to show that the family ${\cal H}$ is bounded and 
satisfies the assumptions (1),(3) in the statement
of Theorem \ref{hypextension}.

Lemma \ref{rbounded} shows 
that the family ${\cal H}=\{H_{X}\mid X\}$ is bounded.

By Lemma \ref{identify}, $H_{X}$ is isometric
to ${\cal A}(X,1)$ and hence by Lemma \ref{inductbegin},
$H_X$ is $\delta$-hyperbolic 
for a number $\delta >0$ not depending on $X$.
The bounded penetration property for ${\cal H}$ 
follows from property c) above, applied to 
thick subsurfaces of Euler characteristic $\chi(S)+p$
and quasi-geodesics in ${\cal A}(p-1)$ (compare \cite{H13}).
Thus by Theorem \ref{hypextension} and the induction
hypothesis, ${\cal A}(p)$ is hyperbolic, and level $p$ hierarchy
paths are uniform quasi-geodesics in ${\cal A}(p)$.

We are left with verifying property c) above for ${\cal A}(p)$.
By Lemma \ref{good2}, 
this property holds true for level $p$ hierarchy paths 
with the number $\kappa_p>0$ replacing $\xi(L,p)$.
The argument in the proof of Lemma 6.5 of \cite{H13}
then yields this property for an arbitrary  
$L$-quasi-geodesic in ${\cal A}(p)$ for a suitable number
$\xi(L,p)>0$.

Namely, by hyperbolicity, for every $L>1$ there is a
number $n(L)>1$ so that for every $L$-quasi-geodesic
$\eta:[0,k]\to {\cal A}(p)$ of finite length, the Hausdorff distance
between the image of $\eta$ and the image of a 
level $p$ hierarchy path $\gamma$ with the same endpoints
does not exceed $n(L)$.

Let $Y\subset S$ be a thick subsurface of Euler characteristic
$\chi(Y)\geq \chi(S)+p$. Assume that
\begin{equation}\label{assumption}
{\rm diam}(\pi_Y(\eta(0)\cup \eta(k)))\geq 
2\kappa_p +L(4n(L)+10).\end{equation}
By the properties of level $p$ hierarchy paths,
if $\tilde \gamma$ denotes the canonical modification of 
$\gamma$ then  
there is some $u\in \mathbb{Z}$ so that 
$\tilde \gamma(u)\in A$ where $A\subset {\cal C\cal G}$ is 
the set of all curves which are disjoint from $Y$.

By the choice of $n(L)$, the quasi-geodesic $\eta$ 
passes through the 
$n(L)$-neighborhood of $Y$. By this we mean that
there is a vertex $x$ on $\eta$ and a simplicial 
path in ${\cal A}(p)$ 
of length at most $n(L)$ which connects $x$ to
a non-separating simple closed curve  contained in $Y$.

Let $s_0+1\leq t_0-1$ be the smallest and biggest number,
respectively, so that 
$\eta(s_0+1),\eta(t_0-1)$ are contained in the
$n(L)$-neighborhood of $Y$ in the sense defined
in the previous paragraph.   
The distance in ${\cal A}(p)$ between
$\eta(s_0)$ and $\eta(t_0)$ does not exceed 
$2n(L)+3$.

A level $p$ hierarchy path connecting $\eta(0)$ to 
$\eta(s_0)$ 
is contained in the $(n(L)-1)$-neighborhood of
$\eta[0,s_0]$ and hence its canonical modification does not
pass through $A$. Similarly, a canonical modification of a 
level $p$ hierarchy path
connecting $\eta(t_0)$ to $\eta(k)$
does not pass through $A$. By the assumption (\ref{assumption}), 
and the properties of hierarchy paths, this implies that
\[{\rm diam}(\pi_Y(\eta(s_0)\cup\eta(t_0)))\geq 
L(4n(L)+10).\] 

The distance in ${\cal A}(p)$ between 
$\eta(s_0),\eta(t_0)$
is at most $2n(L)+3$, and hence since
$\eta$ is an $L$-quasi-geodesic, the length of the segment
$\eta[s_0,t_0]$ is at most $L(2n(L)+4)$. Then the length of a
canonical modification $\tilde \eta[s,t]$ of $\eta[s_0,t_0]$ is at most
$L(4n(L)+8)$. 
Now if $c,d$ are disjoint simple closed curves
which intersect $Y$ then the diameter of $\pi_Y(c\cup d)$ is at
most one. Thus if $\tilde \eta(\ell)$ intersects $Y$ for all 
$\ell$ then
\[{\rm diam}(\pi_Y(\tilde \eta(s)\cup \tilde \eta(t)))=
{\rm diam}(\pi_Y(\eta(s_0)\cup\eta(t_0)))\leq L(4n(L)+8 )\]
which is a contradiction.

This completes the induction step and proves
Proposition \ref{curves}.
\qed

\bigskip

The arguments in \cite{H13} can now be used without
modification to identify the Gromov boundary of 
${\cal N\cal C}(S,1)$. To this end 
let ${\cal L}$ be the set of all geodesic laminations on
$S$ equipped with the coarse Hausdorff topology.
In this topology, a sequence $(\nu_i)$ converges to 
$\nu$ if any limit in the usual Hausdorff topology
of a convergent subsequence contains $\nu$ as a sublamination.

For each thick subsurface $X$ of $S$ let 
${\cal L}(X)\subset {\cal L}$ be the set of all minimal geodesic
laminations which fill up $X$, equipped with the
coarse Hausdorff topology. We have

\begin{corollary}\label{boundary}
The Gromov boundary of ${\cal N\cal C}(S,1)$ equals
$\cup_X{\cal L}(X)$ equipped with the coarse
Hausdorff topology.
\end{corollary}

\section{Proof of  the theorem}

In this section we consider 
an oriented surface $S$ of genus
$g\geq 2$ with $m\geq 0$ punctures.
In the introduction we defined for 
$n\geq 1$ the graph ${\cal N\cal C}(S,n)$ of non-separating
multicurves in $S$ with $n$ components. 
Our goal is to show

\begin{theorem}\label{cutsystem}
For $n<g/2+1$ the graph ${\cal N\cal C}(S,n)$ is hyperbolic.
\end{theorem}

The case $n=1$ 
is just Proposition \ref{curves}.
For $2\leq n<g/2+1$ we use induction on $n$ 
similiar to the arguments in the 
proof of Proposition \ref{curves}.
There are no new tools needed, however all 
the constructions have to be adjusted to the situation at hand.




We begin with describing an electrification 
of the graph ${\cal N\cal C}(S,n)$. 
First, 
for a non-separating
$(n-1)$-multicurve $\nu\in {\cal N\cal C}(S,n-1)$ 
let $H_\nu$ 
be the complete subgraph of ${\cal N\cal C}(S,n)$ 
whose vertex set consists of all non-separating
$n$-multi-curves containing $\nu$.
We have

\begin{lemma}\label{identifyhnu}
There is a natural graph isomorphism
$H_\nu\to {\cal N\cal C}(S-\nu,1)$.
\end{lemma}
\begin{proof} If $\beta\in H_\nu$ is any 
non-separating $n$-multicurve containing $\nu$ 
then $\beta-\nu$ is a non-separating simple closed curve
in $S-\nu$. If $\beta,\beta^\prime\in H_\nu$ are connected
by an edge then
the non-separating simple closed curves
$\beta-\nu$ and 
$\beta^\prime-\nu$ are disjoint and hence they
are connected in ${\cal N\cal C}(S-\nu,1)$ by an edge. 

Vice versa, the union with $\nu$ of any non-separating
simple closed 
curve $c$ in $S-\nu$ is a non-separating multicurve in 
$H_\nu$. If $c^\prime\subset S-\nu$ is non-separating
and disjoint from $c$ then $\nu\cup c$ and $\nu\cup c^\prime$
are connected by an edge in $H_\nu$.

This shows the 
lemma.
\end{proof}


Let ${\cal H}=\{H_\nu\mid \nu\in {\cal N\cal C}(S,n-1)\}$.

\begin{lemma}\label{quasiiso}
${\cal N\cal C}(S,n-1)$ is quasi-isometric to the
${\cal H}$-electrification of ${\cal N\cal C}(S,n)$.
\end{lemma}
\begin{proof} Let ${\cal E}$ be the ${\cal H}$-electrification
of ${\cal N\cal C}(S,n)$.
Define a vertex embedding $\Lambda:
{\cal N\cal C}(S,n-1)\to {\cal E}$ by associating to 
a non-separating $(n-1)$-multicurve $c$ any 
non-separating $n$-multicurve
$\Lambda(c)$ containing $c$. We claim that 
$\Lambda$ is coarsely $8$-Lipschitz.

To see this 
let $c_0,c_1$ be connected by an edge in ${\cal N\cal C}(S,n-1)$.
Then $c_1$ is obtained from 
$c_0$ by removing a component $a$ from $c_0$ and
replacing it by a component $b$ disjoint from $c_0$.

The union $c_0\cup b$ is a multicurve with $n$
components. If this multicurve is  
non-separating then we can view
it as a vertex $x\in{\cal E}$. Since $\Lambda(c_0)$
is a non-separating $n$-multicurve containing $c_0$,
the distance in ${\cal E}$ 
between $\Lambda(c_0)$ and $x$ equals at most $2$.
Similarly, the distance in ${\cal E}$ between $\Lambda(c_1)$ and
$x$ is at most $2$ and hence the distance in ${\cal E}$
between $\Lambda(c_0)$ and $\Lambda(c_1)$ is at most $4$.

If $c_0\cup b$ is not non-separating then 
$a\cup b$ is a bounding pair in 
$S-(c_0-a)=S-(c_1-b)$. 
Choose a non-separating simple closed 
curve $\omega\in S-(c_0-a)$ which is  disjoint from $a\cup b$ 
so that both $c_0\cup \omega$ and
$c_1\cup \omega$ are non-separating.
Such a curve exists since the genus $g-n+2$ of 
$S-(c_0\cup a)$ is at least three.
Apply the same argument to the non-separating 
$(n-1)$-multicurves $c_0,(c_0-a)\cup \omega$ and 
to $(c_0-a)\cup \omega,c_1$. We conclude that the
distance in ${\cal E}$ between $\Lambda(c_0)$ and 
$\Lambda(c_1)$ is at most $8$.

On the other hand, a map which associates to 
a vertex $x\in {\cal N\cal C}(S,n)\subset 
{\cal E}$ an $(n-1)$-multicurve 
contained in $x$ is a coarsely Lipschitz
coarse inverse of $\Lambda$. Thus indeed 
$\Lambda$ is a quasi-isometry.
%
\end{proof}

By Lemma \ref{identifyhnu}, for each vertex
$\nu\in {\cal N\cal C}(S,n-1)$ the complete connected
subgraph
$H_\nu$ of ${\cal N\cal C}(S,n)$
is quasi-isometric to the hyperbolic graph
${\cal N\cal C}(S-\nu,1)$.  However,
by the results in Section 3, the graph
${\cal N\cal C}(S-\nu,1)$ is not
quasi-isometric to the curve graph of $S-\nu$.
Thus controlling distances in these subgraphs
via subsurface projection is not immediate.
Moreover, for a subsurface $X$ of $S$ of genus 
$g-n+1$ there are in general many different
non-separating $(n-1)$-multicurves disjoint from $X$.

To resolve this problem we use exactly the strategy
from Section 3. Namely, we introduce intermediate
graphs ${\cal N\cal C}(S,n,p)$ $(p\geq 1)$ which 
are defined as follows.
Vertices of ${\cal N\cal C}(S,n,p)$ are
non-separating $n$-multicurves. Two such 
multicurves $\nu_0,\nu_1$ are connected
by an edge of length one if $\hat \nu=\nu_0\cap \nu_1$
is an $(n-1)$-multicurve and if the non-separating
simple closed curves $a=\nu_0-\hat \nu$ and 
$b=\nu_1-\hat \nu$ are connected by
an edge in the graph ${\cal A}(S-\hat \nu,p)$.

The strategy is now to deduce hyperbolicity
of ${\cal N\cal C}(S,n,1)$ from hyperbolicity
of ${\cal N\cal C}(S,n-1)$, and for $p\geq 2$ to deduce
hyperbolicity of ${\cal N\cal C}(S,n,p)$ from
hyperbolicity of ${\cal N\cal C}(S,n,p-1)$.

For a non-separating $(n-1)$-multicurve
$\nu\in {\cal N\cal C}(S,n-1)$ let
$H_\nu(1)$ be the complete subgraph of 
${\cal N\cal C}(S,n,1)$ whose vertex
set consists of all non-separating 
$n$-multicurves containing $\nu$. 
Define moreover 
\[{\cal H}(1)=\{H_\nu(1)\mid \nu\}.\]
The following 
is immediate from the reasoning in Lemma \ref{identifyhnu}
and Lemma \ref{quasiiso}.

\begin{lemma}\label{identify2}
\begin{enumerate}
\item
There is a natural graph isomorphism 
$H_\nu(1)\to {\cal A}(S-\nu,1)$.
\item ${\cal N\cal C}(S,n-1)$ is quasi-isometric to the
${\cal H}(1)$-electrification of ${\cal N\cal C}(S,n,1)$.
\end{enumerate}
\end{lemma}

Our first goal is to apply
Theorem \ref{hypextension} to the family
${\cal H}(1)$ of subgraphs of ${\cal N\cal C}(S,n,1)$ 
to deduce hyperbolicity
of ${\cal N\cal C}(S,n,1)$ from hyperbolicity
of ${\cal N\cal C}(S,n-1)$. 
To this end
we have to check that the assumptions in the theorem
are satisfied. 

For $\nu\not=\zeta\in {\cal N\cal C}(S,n-1)$,
the vertex set of  
the intersection $H_\nu(1)\cap H_\zeta(1)$ is 
the set of all non-separating 
$n$-multicurves which contain both 
$\nu$ and $\zeta$ and hence it consists of  
at most one point. Thus ${\cal H}(1)$ is bounded.

By the first part of Lemma \ref{identify2} and
by Lemma \ref{inductbegin}, for every
$\nu\in {\cal N\cal C}(S,n-1)$ 
the graph 
$H_\nu(1)$ is $\delta$-hyperbolic for a number $\delta>0$ which 
does not depend on $\nu$.

The final step is the verification of the
bounded penetration property, which is more involved.

Let ${\cal E}$ be the ${\cal H}(1)$-electrification of 
${\cal N\cal C}(S,n,1)$. 
Let $\beta:[0,k]\to {\cal E}$ 
be an efficient simplicial quasi-geodesic. 
If the integer $i<k$ is such that $\beta(i),\beta(i+1)\in 
{\cal N\cal C}(S,n,1)$ then $\beta(i)$ and $\beta(i+1)$
are $n$-multicurves, and 
$\nu=\beta(i)\cap \beta(i+1)$
is an $(n-1)$-multicurve such that $\beta(i),
\beta(i+1)\in H_\nu(1)$.
If $\beta(i)=v_\nu$ is a special 
vertex defined by an $(n-1)$-multicurve $\nu$ 
then $\beta(i-1),\beta(i+1)\in H_\nu(1)$.

As in Section 2, an
enlargement of $\beta$ is a path
$\hat \beta:[0,m]\to {\cal N\cal C}(S,n,1)$
defined as follows.
For each $i$ such that $\beta(i)=v_\nu$ for some 
$\nu\in {\cal N\cal C}(S,n-1)$, 
replace the arc $\beta[i-1,i+1]$ by a path of the form
$j\to \nu\cup \zeta(j)$ where
$\zeta$ is a simplicial geodesic in ${\cal A}(S-\nu,1)$ connecting
$\beta(i-1)-\nu$ to $\beta(i+1)-\nu$.

By induction, we now assume that for every
$L>1$ there is number $\kappa^\prime(L)>0$ with the following
property. 
Let $X\subset S$ be a connected subsurface of genus
$g/2<h\leq g-n+1<g$. If 
$\alpha:[0,\ell]\to {\cal N\cal C}(S,n-1)$ 
is an $L$-quasi-geodesic with the property
that $\pi_X(\alpha(i))\not=\emptyset$ for all $i$ then 
the diameter of $\pi_X(\cup_i\alpha(i))$ 
in the curve graph of $X$ does not exceed
$\kappa^\prime(L)$.
Note that the case $n=2$ 
holds true by 
the results in Section 3.

Lemma \ref{quasiiso} then implies that 
for every $L>1$ there is a
number $\kappa(L)>3L$ with the following property.
Let $\beta:[0,k]\to {\cal E}$ be an 
efficient $L$-quasi-geodesic.
If $X\subset S$ is a connected subsurface of genus 
$g/2<h\leq g-n+1$ so that the diameter
of $\pi_X(\beta(0)\cup \beta(k))$ in the curve graph of $X$ is 
at least $\kappa(L)$, then there is 
an $(n-1)$-multicurve $\nu\in {\cal N\cal C}(S,n-1)$ 
disjoint from $X$, and there is some
$i<k$ so that
$\beta(i)\in H_\nu(1)\subset {\cal N\cal C}(S,n,1)$.

As in Section 3,
an enlargement $\hat\beta$ of $\beta$ 
admits a \emph{canonical modification} $\tilde\beta$ 
as follows. If $\beta(i)=\nu\cup a,
\beta(i+1)=\nu\cup b$ are such that the simple closed
curves 
$a,b$ are not disjoint but contained in a thick subsurface
$X$ of $S-\nu$ of Euler characteristic
$\chi(X)=\chi(S)+1=\chi(S-\nu)+1$, then replace the edge
$\beta[i,i+1]$ in ${\cal N\cal C}(S,n,1)$ 
by an edge-path $\zeta[j-1,j+1]$ 
of length two
in the space of (not necessarily non-separating)
$n$-multicurves so that $\zeta(j-1)= 
\beta(i)=\nu\cup a,\zeta(j+1)=\beta(i+1)=\nu\cup b$
and that 
$\zeta(j)=\nu\cup c$ for a (perhaps
separating) simple closed curve $c\subset S-\nu$ 
which is disjoint from $a,b$ and $X$.

The following can
now be derived 
from the results in Section 2, Lemma \ref{identifyhnu} and 
Lemma \ref{quasiiso}.

\begin{proposition}\label{intermediate}
\begin{enumerate}
\item The graph ${\cal N\cal C}(S,n,1)$ is hyperbolic.
\item For every $L>1$ there are numbers $L^\prime>1$,
$\kappa(L)>0$ with the 
following property. 
Let $\hat\beta:[0,k]\to {\cal N\cal C}(S,n,1)$
be an enlargement of an efficient $L$-quasi-geodesic 
in ${\cal E}$.
\begin{enumerate}
\item $\hat\beta$ 
is an $L^\prime$-quasi-geodesic
in ${\cal N\cal C}(S,n,1)$. 
\item Let $X\subset S$ be a connected subsurface
of genus $h\in (g/2,g-n+1]$ and Euler characterisitic
$\chi(X)=\chi(S)+1$. If
${\rm diam}(\pi_X(\beta(0)\cup \beta(k)))\geq \kappa(L)$
then a canonical modification of 
$\hat\beta$ passes through the complement of $X$.
\end{enumerate}
\end{enumerate}
\end{proposition}
\begin{proof} By the results in Section 2, by
Lemma \ref{identifyhnu} and Lemma \ref{quasiiso},
to show hyperbolicity of ${\cal N\cal C}(S,n,1)$  
we only have to show the bounded penetration
property for efficient quasi-geodesics in ${\cal E}$
and the family of subgraphs ${\cal H}(1)$.

To this end let $\beta:[0,k]\to {\cal E}$ be an efficient  
$L$-quasi-geodesic in ${\cal E}$ and let 
$\nu\in {\cal N\cal C}(S,n-1)$.
Assume that
\[{\rm diam}(\pi_{S-\nu}(\beta(0)\cup\beta(k)))\geq 3\kappa(L).\]
Then by induction hypothesis,
$\beta$ passes through 
$H_\nu(1)$. 

The diameter of $H_\nu(1)$ in ${\cal E}$ 
equals two.
Thus if $0\leq i\leq j\leq k$ are the first and last
points, respectively, of the intersection of $\beta$ with
$H_\nu(1)$, then the length $j-i$ of $\beta[i,j]$ is 
at most $3L<\kappa(L)$.

An application of the induction hypothesis
to $\beta[0,i]$ and to $\beta[j,k]$ shows that
\[{\rm diam}(\pi_X(\beta(0)\cup \beta(i)))\leq \kappa(L)
\text{  and }
{\rm diam}(\pi_X(\beta(j)\cup \beta(k)))\leq \kappa(L)\]
and therefore 
${\rm diam}(\pi_{X}(\beta(i)\cup \beta(j)))\geq \kappa(L)>3L$.

Hence $\beta[i,j]$ passes through
the special vertex $v_\nu$.
Moreover, 
the points $\beta(i),\beta(j)$ 
are of the form $\nu\cup a(i),\nu\cup a(j)$ for 
non-separating simple closed curves 
$a(i),a(j)\subset S-\nu$ such that
${\rm diam}(\pi_{S-\nu}(\beta(0)\cup a(i)))\leq \kappa(L)$ 
and 
${\rm diam}(\pi_{X}(\beta(k)\cup a(j)))\leq \kappa(L)$.
This immediately implies the bounded
penetration property for the regions $H_\nu(1)$ 
and completes the proof of 
hyperbolicity of ${\cal N\cal C}(S,n,1)$.

The other statements are also immediate from the induction hypothesis
and Theorem \ref{hypextension}.
\end{proof}

Next we show that hyperbolicity of 
${\cal N\cal C}(S,n,p-1)$ implies 
hyperbolicity of ${\cal N\cal C}(S,n,p)$.
To this end we proceed as in Section 3. 
The proofs are completely analogous to the 
proofs in Section 3.

Once more,
our goal is to apply Theorem \ref{hypextension}.
For this denote for a connected subsurface $X$ of $S$ of genus
$g-n+1$ and Euler characteristic $\chi(X)=\chi(S)+p-1$
by $B_X$ 
the complete subgraph of   
${\cal N\cal C}(S,n,p-1)$ 
whose vertices consist of 
all non-separating
$n$-multicurves $\nu$ which are disjoint from the 
boundary of $X$.
By assumption on the genus
of $X$, every non-separating $n$-multicurve 
$\nu$ intersects
$X$ and therefore a vertex $\nu\in B_X$ has
at least one component which is contained in $X$.

In the next lemma, the assumption
$n-1<g/2$ is used in an essential way.

\begin{lemma}\label{subsurfacetwo}
There is a number $L_0>1$ not depending on $X$ 
such that 
$B_X$ is $L_0$-quasi-isometric to ${\cal A}(X,1)$.
\end{lemma}
\begin{proof} 
Let $\nu_0\subset S-X$ be a non-separating $(n-1)$-multicurve.
Then for every simple closed non-separating
curve $b\subset X$, the union $\nu_0\cup b$ is a 
vertex in ${\cal N\cal C}(S,n,p)$. Moreover, 
by Lemma \ref{identify}, two such 
pairs $\nu_0\cup a,\nu_0\cup b$ 
are connected by an edge in ${\cal N\cal C}(S,n,p)$ if and only if 
$a,b$ are connected by an edge in ${\cal A}(X,1)$.

Now let $\zeta\in B_X$ be arbitrary. Then the number
$k$ of components of $\zeta$ contained in $X$ 
is non-zero. Let $c_1,\dots,c_k$ be these components.
Then $c_1\cup\dots\cup c_k$ is non-separating
$k$-multicurve contained in $X$. 

Choose a non-separating
$n$-multicurve $\zeta^\prime\subset c_1\cup\dots \cup c_k$
contained in $X$. 
Such a multicurve exists since $n <h$.
By the definition of $B_X$, the distance
in $B_X$ between $\zeta$ and $\zeta^\prime$ equals $n-k$.
Moreover, the distance in $B_X$ between $\zeta^\prime$ and 
and a non-separating $n$-multicurve containing
$\nu_0$ equals $n-1$. But this just means that
the subspace of $B_X$ of non-separating $n$-multicurves
containing $\nu_0$ is coarsely dense. 
Consequently, associating to a vertex $\zeta\in B_X$ a
component of $\zeta$ contained in $X$ defines a Lipschitz
map $B_X\to {\cal A}(X,1)$. 

That this map coarsely
does not decrease distances follows from the
fact that any vertex $\zeta\in B_X$ contains
at least one component which is contained in $X$, and
adjacent vertices contain components which either
are disjoint or contained in a common thick
subsurface of $X$ of Euler characteristic 
$\chi(X)+1$.
\end{proof}

Define a family of subgraphs 
\[{\cal B}(p)=\{B_X\mid X\}\] 
of ${\cal N\cal C}(S,n,p)$ 
where $X$ runs through
the subsurfaces of $S$ of genus $g-n+1$ and Euler characteristic
$\chi(X)=\chi(S)+p-1$. 
By Lemma \ref{subsurfacetwo} and Lemma \ref{inductbegin}, 
each of the graphs $B_X$ is quasi-isometric
to the curve graph of $X$, in particular it is hyperbolic. 

We claim that the family ${\cal B}(p)$ is bounded.
To this end let $X,Y$ be two subsurfaces of $S$ of the same
genus $g-n+1$ and the same Euler characteristic
$\chi(S)+p-1$. Let $\nu\in B_X\cap B_Y$; then
$\nu$ is a non-separating $n$-multicurve disjoint from the boundaries
of both $X,Y$. Since the genus of $X,Y$ equals
$g-n+1$, at least one component of $\nu$ is contained
in $X\cap Y$. However, since $X\not=Y$, $X\cap Y$ is 
a proper subsurface of $X$. By Lemma \ref{subsurfacetwo}
and Lemma \ref{inductbegin}, 
$B_X$ is quasi-isometric to the curve graph of $X$
and hence the diameter of $B_X\cap B_Y$ is
uniformly bounded.

Lemma \ref{subsurfacetwo} shows that each of the graphs
$B_X$ is $\delta$-hyperbolic for a number
$\delta >0$ not depending on $X$. Moreover,
by the definition of the graphs ${\cal N\cal C}(S,n,p)$,
the graph ${\cal N\cal C}(S,n,p-1)$ is quasi-isometric to 
the ${\cal B}$-electrification of ${\cal N\cal C}(S,n,p-1)$.
Thus for an application of Theorem \ref{hypextension}, we are
left with showing the bounded penetration property.

However, this property follows by the induction 
assumption on subsurface projection and Lemma \ref{subsurfacetwo}.

\bigskip
{\bf Example:} The following example was observed by
Tarik Aougab and  
Saul Schleimer and shows that the bound $n<g/2+1$ in 
Theorem \ref{cut} is sharp.

Namely, let $S$ be a closed surface of genus $4$ and let
$d$ be a separating simple closed curve which 
decomposes $S$ into two surfaces $X_1,X_2$ of genus
$2$ with one boundary component. 
Let $\phi_i$ be a pseudo-Anosov element in 
the mapping class group of $X_i$ and let
$a_i$ be a non-separating 2-multicurve 
in $X_i$ $(i=1,2)$.
Let moreover $c$ be a non-separating simple closed
curve which is disjoint from $a_i$ and 
intersects $d$ in two points so that
$a_1,a_2,c$ defines a non-separating
$5$-multicurve $\nu$.

For all $k,\ell\in \mathbb{Z}$ the pair 
$(\phi_1^k,\phi_2^\ell)$ defines a reducible mapping class for $S$,
moreover it is easy to see that the 
distance in ${\cal N\cal C}(S,5)$ between
$\nu$ and $(\phi_1^k,\phi_2^\ell)\nu$ is comparable to $k+\ell$.
The reason is that the subsurface projection
of any non-separating $5$-multicurve in $S$ into 
each of the subsurfaces $X_1,X_2$ does not vanish. 
But this just means that
${\cal N\cal C}(S,5)$ contains
a quasi-isometrically embedded $\mathbb{R}^2$.

\bigskip

Define a subsurface $Y$ of $S$ to be \emph{$n$-heavy} 
if the genus of 
$Y$ is at least $g-n+1$.
Let ${\cal L}(Y)$ be the set of all minimal geodesic laminations
which fill up $Y$.
Similarly to Corollary \ref{boundary} we have

\begin{corollary}\label{boundary2}
The Gromov boundary of ${\cal N\cal C}(S,n)$ equals 
$\cup_Y{\cal L}(Y)$ equipped with the coarse Hausdorff topology
where $Y$ passes through the $n$-heavy subsurfaces of $S$.
\end{corollary}

\noindent
{\bf Remark:} The main result in this note can also
be obtained with the tools developed in \cite{MS13}. 
To the best of our knowledge, these tools do not have
any obvious advantage over the tools we used.

\noindent
MATHEMATISCHES INSTITUT DER UNIVERSIT\"AT BONN\\
ENDENICHER ALLEE 60,\\ 
53115 BONN, GERMANY\\

\smallskip

\noindent e-mail: ursula@math.uni-bonn.de

\end{document}